\documentclass[10pt,twoside,a4paper,reqno]{amsart}
\usepackage{amsmath,amscd,amsfonts,amsthm,amssymb,latexsym}

\theoremstyle{plain}
   \newtheorem{theorem}{Theorem}[section]

   \newtheorem{corollary}[theorem]{Corollary}
   \newtheorem{problem}{Problem}
   
   \newtheorem*{question}{Question}
\theoremstyle{definition}

\theoremstyle{remark}
   \newtheorem{remark}{Remark}[section]

\newcommand {\bC} {\mathbb {C}}
\newcommand {\bR} {\mathbb {R}}
\newcommand {\bN} {\mathbb {N}}

\newcommand {\bs} {\mathbf {s}}
\newcommand {\bt} {\mathbf {t}}
\newcommand {\bx} {\mathbf {x}}
\newcommand {\by} {\mathbf {y}}
\newcommand {\bz} {\mathbf {z}}
\newcommand {\ba} {\mathbf {a}}
\newcommand {\bo} {\mathbf {0}}
\newcommand {\bone} {\mathbf {1}}

\newcommand {\al} {\alpha}
\newcommand {\be} {\beta}

\newcommand {\la} {\lambda}

\newcommand {\cA} {\mathcal {A}}

\newcommand {\calD} {\mathcal {D}}

\newcommand {\pa} {\partial}

\newcommand{\cP}{POS}
\newcommand{\cN}{\overline{POS}}
\newcommand{\cS}{SOS}
\newcommand{\cE}{ELL}

\newcommand{\cH}{\mathcal{H}}
\newcommand{\cF}{\mathcal{F}}

\newcommand{\fL}{\mathfrak{L}}
\newcommand{\fM}{\mathfrak{M}}

\def\newop#1{\expandafter\def\csname #1\endcsname{\mathop{\rm
#1}\nolimits}}

\newop{supp}

\begin{document}

\numberwithin{equation}{section}

\title[Positive, non-negative and elliptic polynomials]{Classifications of 
linear operators preserving elliptic, positive and non-negative polynomials}
\author[J.~Borcea]{Julius Borcea}
\address{Department of Mathematics, Stockholm University, SE-106 91 Stockholm,
   Sweden}
\email{julius@math.su.se}



\thanks{The author is a Royal Swedish Academy of Sciences 
Research Fellow supported by a grant from the Knut and Alice Wallenberg 
Foundation.}

\begin{abstract}
We characterize all linear operators on finite or infinite-dimensio\-nal 
spaces 
of univariate real polynomials preserving the sets of elliptic, positive, 
and non-negative polynomials, respectively. This is done by means of 
Fischer-Fock dualities, Hankel forms, and convolutions with 
non-negative measures. We also establish higher-dimensional analogs of 
these results. 
In particular, our 
classification theorems solve the questions raised in 
\cite{BGS} originating from entire function theory and the 
literature pertaining to Hilbert's 17th problem.
\end{abstract}

\maketitle

\section{Introduction}\label{s-in}

The dynamics of zero sets of polynomials and transcendental entire functions 
under linear transformations is a central topic in geometric function theory 
with applications ranging from statistical mechanics and probability theory 
to combinatorics, analytic number theory, and matrix theory; see, e.g., 
\cite{BB-LY}--\cite{BBL} and the references therein. Describing linear 
operators that preserve non-vanishing properties is a basic question that 
often turns out to be quite difficult. For instance, the problem  -- going 
back to Laguerre and P\'olya-Schur 
\cite{PS} -- of 
characterizing linear operators preserving univariate real polynomials with 
{\em all real zeros}
remained unsolved until very recently \cite{BB-A}. Just as fundamental 
are the problems -- originating from entire function theory 
\cite{CC,CC1,CC2,H,PSz} and the 
literature pertaining to Hilbert's 
17th problem \cite{H-P,PD,R} -- of classifying linear operators preserving  
the set of 
univariate real polynomials with {\em no real zeros} or the closely related 
sets of positive, respectively non-negative polynomials. Let us formulate 
these problems explicitly.

Given $d\in \bN_0$, where $\bN_0$ is the additive  
semigroup of non-negative integers,
let $\bR_d[x]=\{p\in\bR[x]:\deg(p)\le d\}$ and set
\begin{equation*}
\begin{split}
&\cP=\{p\in\bR[x]:p(x)>0,\,x\in\bR\},\quad \cP_d=\cP\cap\bR_d[x],\\
&\cS=\left\{\sum_{i=1}^{r}p_i(x)^2:p_i\in\bR[x],\,1\le i\le r<\infty\right\},
\quad \cS_d=\cS\cap\bR_d[x],\\
&\cE=\{p\in\bR[x]: p(x)\neq 0,\,x\in\bR\},\quad \cE_d=\cE\cap\bR_d[x].
\end{split}
\end{equation*}
It is well-known that
$\cS=\cN:=\{p\in\bR[x]:p(x)\ge0,\,x\in\bR\}$, 
$\cE=\cP\cup(-\cP)$, and $\cP=\cup_{a>0}(\cS+a)$.
Polynomials in $\cP,\cS,\cE$ are called {\em positive}, 
{\em non-negative} or 
{\em sos} (sums of squares), and {\em elliptic}, respectively. 
Let $\fL_d$ be the set of all linear operators 
$T:\bR_d[x]\to\bR[x]$ and $\fL$ be the semigroup of all linear 
operators $T:\bR[x]\to\bR[x]$. For $V\in\{\cP,\cS,\cE\}$ and $d\in\bN_0$ 
define
$$
\fL(V)=\{T\in\fL:T(V)\subseteq V\},\quad 
\fL_d(V)=\{T\in\fL_d:T(V_d)\subseteq V\}.
$$ 

\begin{problem}\label{prob1}
Characterize $\fL_d(\cP)$, $\fL_d(\cS)$, $\fL_d(\cE)$ for each 
even $d\in\bN_0$.
\end{problem}

\begin{problem}\label{prob2}
Characterize the semigroups $\fL(\cP)$, $\fL(\cS)$, $\fL(\cE)$.
\end{problem}

In physics it is useful to distinguish between between ``soft'' and ``hard'' 
theorems asserting the non-vanishing of partition functions in certain 
regions.\footnote{This dichotomy stems 
from ``soft-core'' vs.~``hard-core'' pair interactions in lattice-gas models 
and does not refer to the level of difficulty in proving such theorems but 
rather to the fact that in some sense ``soft'' theorems are constraint-free 
while ``hard'' theorems involve constraints of various kinds such as the 
maximum degree of a graph; see, e.g., \cite{BB-LY} and the 
references therein.} 
By analogy with this terminology and the one used in the classification 
of linear preservers of real-rooted polynomials \cite{BB-LY,BB-A}, one may 
say that results pertaining to 
Problem~\ref{prob1} are ``hard'' or ``algebraic'' (bounded degree) while 
those for Problem~\ref{prob2} are ``soft'' or ``transcendental''  
(unbounded degree). 
Special cases of Problems \ref{prob1}--\ref{prob2} were recently studied 
in \cite{BGS}, where one can also find a list of partial results known so far 
and corrections to some erroneous claims from the existing literature.

In this paper we give complete solutions to 
Problems \ref{prob1}--\ref{prob2}. 
In particular, this supersedes \cite{BGS} 
whose main results are now immediate corollaries of the classification 
theorems in \S \ref{s-uni}. In \S \ref{s-multi} we answer multivariate 
versions of Problems \ref{prob1}--\ref{prob2} while in \S \ref{s-rem} we 
establish some related results and discuss further directions.

One of the tools that we use both in the univariate and the multivariate 
case is the following inner product structure on polynomial spaces. Recall 
(cf., e.g., \cite{BB-P} and the 
references therein) that the {\em Fischer-Fock} or {\em Bargmann-Segal}  
{\em space} $\cF_n$ is the Hilbert 
space of holomorphic functions $f$ on $\bC^n$ such that 
$$
\| f \|^2 := 
\pi^{-n}\int_{\bC^n} 
|f(\bz)|^2 e^{-|\bz|^2} dV(\bz)<\infty,
$$ 
where $\bz=(z_1,\ldots,z_n)\in\bC^n$, $|\bz|^2=\sum_{i=1}^n|z_i|^2$, and 
$dV(\bz)$ is the volume element (Lebesgue measure) in $\bC^n$. 
The inner product in $\cF_n$ is 
given by 
\begin{equation*}
\begin{split}
\langle f, g \rangle  &= 
\pi^{-n}\int_{\bC^n} 
f(\bz)\overline{g(\bz)} e^{-|\bz|^2} dV(\bz) \\
&=f(\pa/\pa z_1,\ldots,\pa/\pa z_n)g(z_1,\ldots,z_n)\Big|_{z_1=\cdots=z_n=0} 
=\sum_{\al\in\bN_0^n}\al!a_\al\overline{b}_\al.
\end{split}
\end{equation*}
Here 
$\al!=\al_1!\cdots\al_n!$ for $\al=(\al_1,\ldots,\al_n)\in\bN_0^n$ while 
$\sum_\alpha a_\al\bz^\alpha$, respectively 
$\sum_\alpha b_\al\bz^\alpha$, is the Taylor expansion of $f$, 
respectively $g$, where $\bz^\al=z_1^{\al_1}\cdots z_n^{\al_n}$. 
Clearly, the space of all polynomials in $n$ variables with real coefficients 
$\bR[z_1,\ldots,z_n]$ is contained in $\cF_n$ and thus inherits a natural 
inner product structure. Given a cone $K\subseteq \bR[z_1,\ldots,z_n]$ we let 
$K^*\subseteq \bR[z_1,\ldots,z_n]$ be its {\em dual} cone, i.e.,
$$
K^*=\{f\in\bR[z_1,\ldots,z_n]: \langle f,g\rangle \ge 0,\,g\in K\}.
$$

The Fischer-Fock space $\cF_n$ was used by Dirac to define second 
quantization and its inner product has since been rediscovered in 
various contexts, for instance in number theory -- where the corresponding 
norm is known as the Bombieri norm -- and the theory of real homogeneous 
polynomials \cite{R}. In $\calD$-module
theory and microlocal Fourier analysis one usually works with 
the inner product on $\cF_n$ defined by 
$(f(\bz),g(\bz))=\langle f(i\bz), g(i\bz)\rangle$. 
For these and further properties 
of $\cF_n$ (such as its Bergman-Szeg\H{o}-Aronszajn reproducing kernel) see, 
e.g., the references in \cite{BB-P}.

Another main ingredient in our characterization theorems is the theory of 
positive definite kernels, Hankel forms, and moment problems in one or 
several variables. There is a vast literature on these topics ranging from
classical works on univariate and multivariate moment problems 
\cite{Ak,C,Ham,Hav,Ka,ST,W} to recent contributions 
\cite{BCR,H-P,Pu-S,Pu-Sm,Pu-V,R} and we refer to these for additional details. 
As usual, a 
$k\times k$ real symmetric matrix $A$ is said to be positive semidefinite 
(respectively, 
positive definite)
if $\bx^t A\bx\ge 0$ for all $\bx\in\bR^k$ (respectively, $\bx^t A\bx> 0$ 
for all $\bx\in\bR^k\setminus \{0\}$). 

Let $\fM_1^+$ denote the class of 
non-negative measures $\mu$ on $\bR$ with all finite moments, that is, 
$\int\!|t^m|d\mu(t)<\infty$, $m\in \bN_0$. Hamburger's theorem \cite{Ham} 
asserts that a sequence of real numbers 
$\{a_m\}_{m\in\bN_0}$ is the moment sequence of 
a measure $\mu\in\fM_1^+$, i.e., 
$$
a_m=\int\!t^m d\mu(t),\quad m\in\bN_0,
$$
if and only if the Hankel matrix $(a_{i+j})_{i,j=0}^{m}$ is positive 
semidefinite for any $m\in\bN_0$. If such a $\mu$ is unique 
the corresponding moment problem is said to be determined.

Haviland's theorem \cite{Hav} is a multivariate analog of Hamburger's 
theorem that may be stated as follows (cf., e.g., \cite[Chap.~6]{BCR}). 
Let $n\in\bN$ and 
$\fM_n^+$ be the class of 
non-negative measures $\mu$ on $\bR^n$ with all finite moments, that is, 
$\int\!|\bt^\al|d\mu(\bt)<\infty$, $\al\in \bN_0^n$, where 
$\bt=(t_1,\ldots,t_n)\in\bR^n$ and $\bt^\al=t_1^{\al_1}\cdots t_n^{\al_n}$ 
for $\al=(\al_1,\ldots,\al_n)\in\bN_0^n$. A sequence of real numbers 
$\{a_\al\}_{\al\in\bN_0^n}$ satisfies
$$
a_\al=\int\!\bt^\al d\mu(\bt),\quad \al\in\bN_0^n,
$$
for some $\mu\in\fM_n^+$ if and only if the following holds: the real-valued 
linear functional $L$ on $\bR[x_1,\ldots,x_n]$ determined by 
$L(\bx^\al)=a_\al$, 
$\al\in\bN_0^n$, satisfies $L(f)\ge 0$ for any non-negative polynomial 
$f\in\bR[x_1,\ldots,x_n]$ (i.e., $f(\bx)\ge 0$, $\bx\in\bR^n$).

\medskip

\noindent
{\em Acknowledgement.~}I would like to thank the anonymous referee for 
interesting remarks and useful suggestions. 

\section{Univariate Polynomials}\label{s-uni}

\subsection{Preliminaries}\label{s-prel}

For convenience, we recall some elementary properties of 
the classes of linear operators under consideration. 
If $T\in\fL_d$ and $T(1)\in\cE$ then the following are 
equivalent: (i) $T\in \fL_d(\cE)$; (ii) $T\in \fL_d(\cP)$ or 
$-T\in \fL_d(\cP)$; (iii) $T\in \fL_d(\cS)$ or $-T\in \fL_d(\cS)$.
Similarly, if $T\in\fL$ and 
$T(1)\in\cE$ then the following are equivalent:
(i$'$) $T\in \fL(\cE)$; (ii$'$) $T\in \fL(\cP)$ or $-T\in \fL(\cP)$;
(iii$'$) $T\in \fL(\cS)$ or $-T\in \fL(\cS)$.
If one drops the assumption $T(1)\in\cE$ then 
(i) $\Leftrightarrow$ (ii) $\Rightarrow$ (iii) and 
(i$'$) $\Leftrightarrow$ (ii$'$) $\Rightarrow$ (iii$'$). Furthermore, 
if $T\in\fL(\cS)$ and $T(1)\equiv 0$ then $T\equiv 0$.
These are simple consequences of straightforward homotopy 
arguments (see, e.g., Remark \ref{r-hom} below and \cite{BGS}).

\subsection{Transcendental Characterizations}\label{s-trans}

We 
start with the ``soft'' (unbounded degree) case and the three 
classification questions 
stated in Problem \ref{prob2}. Recall that any linear operator 
$T\in\fL$ may be uniquely represented as 
a formal power series in $D:=\frac{d}{dx}$ with polynomial coefficients, i.e., 
\begin{equation}\label{w-1}
T=\sum_{i=0}^{\infty}q_i(x)D^i,\quad q_i\in\bR[x],\,i\in\bN_0,
\end{equation}
see, e.g., \cite{BB-P}. Given $T$ as above 
we associate to it a 
one-parameter family of linear differential operators with {\em constant} 
coefficients 
$\{T_y\}_{y\in\bR}$ defined as follows:
\begin{equation*}
T_y=\sum_{i=0}^{\infty}q_i(y)D^i.
\end{equation*}
Note that in general $T_y$ has infinite 
order and $T_y=T$, $y\in\bR$, if $T$ has constant coefficients. 
For each $y\in\bR$ and $m\in\bN_0$ we also define a polynomial 
$p_{y,m}=p_{_{T_y,m}}$ of 
degree at most $m$ and an $(m+1)\times(m+1)$ Hankel matrix 
$\cH_{y,m}=\cH_{_{T_y,m}}$ by 
\begin{equation*}
p_{y,m}(x)=\sum_{i=0}^{m}q_i(y)x^i\in\bR[x],\quad 
\cH_{y,m}=\big((i+j)!q_{i+j}(y)\big)_{i,j=0}^{m}.
\end{equation*}

\begin{remark}\label{r-term}
Clearly, $p_{y,m}(x)\in\bR[x,y]$ for each $m\in\bN_0$. We call this 
polynomial in two variables the $m$-{\em truncated symbol} of the operator 
$T\in\fL$.
\end{remark}

Our first theorem characterizes the semigroup $\fL(\cS)$ of 
linear operators preserving the cone $\cS$ of non-negative$/$sos polynomials.

\begin{theorem}\label{th-s}
Let $T\in\fL$ be as in \eqref{w-1}. The following assertions are equivalent:
\begin{itemize}
\item[(1)] $T\in\fL(\cS)$;
\item[(2)] $T_y\in\fL(\cS)$ for all $y\in\bR$;
\item[(3)] $\langle p_{y,2m},f\rangle\ge 0$ for all $y\in\bR$, $m\in\bN_0$, 
$f\in\cS_{2m}$, i.e., $p_{y,m}\in\cS^*_{2m}$;
\item[(4)] $\cH_{y,m}$ is positive semidefinite 
for all $y\in\bR$, $m\in\bN_0$;
\item[(5)] For any $s\in\bR$ there exists $\mu_s\in\fM_1^+$ such that
$\int\! t^id\mu_s(t)\in \bR[s]$, $i\in\bN_0$, and
$i!q_i(y)=\int\! t^id\mu_y(t)$, $y\in\bR$, $i\in\bN_0$. 
The family of operators $\{T_y\}_{y\in\bR}$ is then given by the convolutions 
$$
T_y(f)(x)=\int\! f(t+x)d\mu_y(t),\quad f\in \bR[x],\,y\in\bR.
$$
\end{itemize}
\end{theorem}

\begin{proof}
Clearly, (2) $\Rightarrow$ (1). If (1) holds then since $\cS$ is invariant 
under shift operators $e^{aD}$, $a\in\bR$, one gets 
$e^{-aD}Te^{aD}\in\fL(\cS)$ hence 
\begin{equation*}
\label{sum}
\sum_{i=0}^{\infty}q_i(x-a)f^{(i)}(x)\ge 0
\end{equation*}
for any $f\in\cS$ and $a,x\in\bR$, which proves (2) (note that for each 
$f\in\bR[x]$ 
the sum in the above inequality is finite). Next, if 
$a,y\in\bR$ and $f\in\bR_d[x]$, $d\in\bN_0$, then
\begin{equation}\label{eq-f}
T_y(f)(a)=\sum_{i=0}^{d}q_i(y)f^{(i)}(a)=\langle p_{y,d},e^{aD}f\rangle.
\end{equation}
Since $f\in\cS_d\Leftrightarrow e^{aD}f\in\cS_d$, $a\in\bR$, identity 
\eqref{eq-f} 
shows that if (2) holds then $\langle p_{y,d}, f\rangle \ge 0$ for all 
$y\in\bR$, $d\in\bN_0$, $f\in\cS_d$. In particular, this is true for all even 
$d$, which proves (3). Conversely, if (3) holds then (2) follows from 
\eqref{eq-f} since any $f\in\cS$ has even degree. 
Now, for $m\in\bN_0$, $y\in\bR$,
and $g(x)=\sum_{i=0}^{m}a_ix^i\in\bR_m[x]$ one has
$$
\langle p_{y,2m},g^2\rangle=\sum_{i,j=0}^{m}(i+j)!q_{i+j}(y)a_ia_j
$$
and thus (3) $\Leftrightarrow$ (4) since any $f\in\cS$ is a sos. 
Finally, by 
Hamburger's theorem (see \S \ref{s-in}) the 
following holds for each $y\in\bR$: $\cH_{y,m}$ is positive semidefinite 
for all $m\in\bN_0$ if 
and only if there exists $\mu_y\in\fM_1^+$ such that 
$i!q_i(y)=\int\! t^id\mu_y(t)$, $i\in\bN_0$. Since all $q_i$ must be 
polynomials we conclude that 
(4) $\Leftrightarrow$ (5), which completes the proof of the theorem.
\end{proof}

\begin{remark}
By Carleman's sufficient condition for the Hamburger moment 
problem to be determined \cite{C}, if $y\in\bR$ is such that 
$\sum_{i=1}^{\infty}[(2i)!q_{2i}(y)]^{-1/2i}=\infty$ then the corresponding 
measure $\mu_y\in\fM_1^+$ in Theorem \ref{th-s} (5) is unique. Recent progress 
on the determinateness of multivariate moment problems has been made in 
\cite{Pu-S,Pu-Sm,Pu-V}.
\end{remark}

The following analog of Theorem \ref{th-s} characterizes the semigroup 
$\fL(\cP)$ of linear operators preserving the cone $\cP$ of positive 
polynomials. 

\begin{theorem}\label{th-p}
Let $T\in\fL$ be as in \eqref{w-1}. The following assertions are equivalent:
\begin{itemize}
\item[(1)] $T\in\fL(\cP)$;
\item[(2)] $T_y\in\fL(\cP)$ for all $y\in\bR$;
\item[(3)] $\langle p_{y,2m},f\rangle> 0$ for all $y\in\bR$, $m\in\bN_0$, 
$f\in\cP_{2m}$; 
\item[(4)] $\cH_{y,m}$ is positive definite for all $y\in\bR$, $m\in\bN_0$;
\item[(5)] $\det(\cH_{y,m})> 0$ for all $y\in\bR$, $m\in\bN_0$;
\item[(6)] For any $s\in\bR$ there exists $\mu_s\in\fM_1^+$ whose support 
$\supp(\mu_s)$ is an infinite set such that
$\int\! t^id\mu_s(t)\in \bR[s]$, $i\in\bN_0$, and
$i!q_i(y)=\int\! t^id\mu_y(t)$, $y\in\bR$, $i\in\bN_0$. 
The family of operators $\{T_y\}_{y\in\bR}$ is then given by the convolutions 
$$
T_y(f)(x)=\int\! f(t+x)d\mu_y(t),\quad f\in \bR[x],\,y\in\bR.
$$
\end{itemize}
\end{theorem}

\begin{proof}
Recall that 
$\cP=\cup_{a>0}(\cS+a)$ and note that if $T\in\fL(\cP)$ then $q_0\in\cP$.
Using these observations and the fact that the cones  
$\cP$, $\cP_{2m}$, $m\in \bN_0$, are invariant under shift operators 
$e^{aD}$, $a\in\bR$, one then gets (1) $\Leftrightarrow$ (2) 
$\Leftrightarrow$ (3) $\Leftrightarrow$ (4) $\Leftrightarrow$ (6) 
as in the proof of Theorem \ref{th-s} (see, e.g., 
\cite[Theorem 1.2]{ST} and Remark 2.3 in \cite[Chap.~6]{BCR} for the 
extra condition that the supports of the representing measures should be 
infinite sets, which is not required in Theorem \ref{th-s} (5)). 
The equivalence (4) 
$\Leftrightarrow$ (5) is a classical result in matrix theory due to Jacobi and 
independently to Hurwitz,
cf., e.g., \cite{Ka,ST}.  
\end{proof}

By the remarks in \S \ref{s-prel} Theorem \ref{th-p} also yields a 
description of the semigroup $\fL(\cE)$ of linear operators preserving the
set of elliptic polynomials. Problem~\ref{prob2} is therefore completely 
solved.

\subsection{Algebraic Characterizations}\label{s-alg}

Since there are no positive, non-negative or elliptic polynomials of odd 
degree, for the ``hard'' (bounded degree) classification questions stated in 
Problem \ref{prob1} it is enough to consider only linear operators 
$T\in\fL_d$ with $d=2k$, $k\in\bN_0$. It is well-known that any such operator 
may be viewed as an 
element of the Weyl algebra $\cA_1(\bR)$ of order at most $d$ (cf., e.g., 
\cite{BB-P}), that is, a 
linear ordinary differential operator with polynomial coefficients of the 
form
\begin{equation}\label{w1}
T=\sum_{i=0}^{d}q_i(x)D^i,\quad q_i\in\bR[x],\,i\in\bN_0.
\end{equation}
By analogy with the ``soft'' case treated in \S \ref{s-trans} we associate 
to $T$ a one-parameter family of 
linear differential operators of order at most $d$ with {\em constant} 
coefficients 
$\{T_y\}_{y\in\bR}$, a one-parameter family of polynomials 
$\{p_{y,d}\}_{y\in\bR}$ in $\bR_d[x]$, and a one-parameter family of 
$(d+1)\times (d+1)$ Hankel matrices $\{\cH_{y,d}\}_{y\in\bR}$ defined as 
follows:
\begin{equation*}
T_y=\sum_{i=0}^{d}q_i(y)D^i,\quad 
p_{y,d}(x)=\sum_{i=0}^{d}q_i(y)x^i\in\bR_d[x],\quad 
\cH_{y,d}=\big((i+j)!q_{i+j}(y)\big)_{i,j=0}^{d}.
\end{equation*}

\begin{remark}\label{r-algs}
Given $d\in\bN_0$ it is clear that $p_{y,d}\in\bR[x,y]$. This two variable 
polynomial is called the {\em algebraic symbol} of the operator $T\in\fL_d$.
\end{remark}

The finite degree versions of Theorem \ref{th-s} and Theorem \ref{th-p} 
given below describe the sets $\fL_d(\cS)$ and $\fL_d(\cP)$, respectively. 
We omit their proofs since they are almost identical to those of 
Theorems \ref{th-s}--\ref{th-p}. 

\begin{theorem}\label{th-sd}
For $d=2k$ and $T\in\fL_d$ as in \eqref{w1} the following are equivalent:
\begin{itemize}
\item[(1)] $T\in\fL_d(\cS)$;
\item[(2)] $T_y\in\fL_d(\cS)$ for all $y\in\bR$;
\item[(3)] $\langle p_{y,d},f\rangle\ge 0$ for all $y\in\bR$, 
$f\in\cS_d$, i.e., $p_{y,d}\in\cS_d^*$;
\item[(4)] $\cH_{y,d}$ is positive semidefinite for all $y\in\bR$.
\end{itemize}
\end{theorem}

\begin{remark}
In his work \cite{Ri} on the truncated 
Hamburger moment problem M.~Riesz showed that the Hankel matrix 
associated to a given finite sequence of real numbers is positive semidefinite 
if and only if the sequence satisfies certain equalities and inequalities 
involving moments of non-negative measures. Using this result one 
can add a fifth condition to 
Theorem \ref{th-sd} (equivalent to those already stated) which is similar 
to Theorem \ref{th-s} (5).
\end{remark}

\begin{theorem}\label{th-pd}
For $d=2k$ and $T\in\fL_d$ as in \eqref{w1} the following are equivalent:
\begin{itemize}
\item[(1)] $T\in\fL_d(\cP)$;
\item[(2)] $T_y\in\fL_d(\cP)$ for all $y\in\bR$;
\item[(3)] $\langle p_{y,d},f\rangle> 0$ for all $y\in\bR$, 
$f\in\cP_d$; 
\item[(4)] $\cH_{y,d}$ is positive definite for all $y\in\bR$;
\item[(5)] $\det(\cH_{y,m})>0$  
for all $y\in\bR$, $0\le m\le d$.
\end{itemize}
\end{theorem}

\begin{remark}
In the case of finite order linear ordinary differential operators 
with constant coefficients Theorems \ref{th-s}--\ref{th-p} reduce to classical 
results of Hurwitz \cite{H} and Remak \cite{Re}, see also P\'olya-Szeg\H{o}'s 
\cite[Chap.~VII]{PSz}. 
\end{remark}

Note that from Theorem \ref{th-pd} and the observations in \S \ref{s-prel}  
one also gets a characterization of the set $\fL_d(\cE)$. Therefore, 
these results fully solve Problem \ref{prob1}.

\section{Multivariate Polynomials}\label{s-multi}

We will now establish multivariate analogs of the characterization 
theorems in \S \ref{s-uni}. Given $n\in\bN$ and  
$\al=(\al_1,\ldots,\al_n)\in\bN_0^n$ let  
\begin{equation*}
\begin{split}
\bR_\al[x_1,\ldots,x_n]&=\{p\in \bR[x_1,\ldots,x_n]: \deg_{x_i}(p)\le \al_i,
\,1\le i\le n\}, \\
\cP(n)&=\{p\in\bR[x_1,\ldots,x_n]:p(x_1,\ldots,x_n)> 0,
\,(x_1,\ldots,x_n)\in\bR^n\},\\
\cN(n)&=\{p\in\bR[x_1,\ldots,x_n]:p(x_1,\ldots,x_n)\ge 0,
\,(x_1,\ldots,x_n)\in\bR^n\},\\
\cS(n)&=\left\{\sum_{i=1}^{r}p_i(x_1,\ldots,x_n)^2:
p_i\in\bR[x_1,\ldots,x_n],\,1\le i\le r<\infty\right\},\\
\cE(n)&=\{p\in\bR[x_1,\ldots,x_n]:p(x_1,\ldots,x_n)\neq 0,
\,(x_1,\ldots,x_n)\in\bR^n\},\\
V_\al&=V\cap\bR_\al[x_1,\ldots,x_n],\, 
V\in\{\cP(n),\cN(n),\cS(n),\cE(n)\}.
\end{split}
\end{equation*}
These sets are the natural multivariate analogs of the ones defined in 
\S \ref{s-in} for $n=1$. Note that $\cS(n)\subseteq \cN(n)$ with strict 
inclusion in general. 
Let $\fL(n)$ be the space of all linear operators 
$T:\bR[x_1,\ldots,x_n]\to \bR[x_1,\ldots,x_n]$. As is well-known (cf., e.g., 
\cite{BB-P}), any 
$T\in\fL(n)$ may be uniquely represented as an 
({\em \`a priori} infinite order) 
linear partial differential operator with polynomial coefficients:
\begin{equation}\label{w-n}
T=\sum_{\al\in\bN_0^n}q_\al(\bx)
\pa^\al,\quad q_\al\in\bR[x_1,\ldots,x_n],\,
\al\in\bN_0^n,
\end{equation}
where $\bx=(x_1,\ldots,x_n)$, 
$\pa^\al=\pa_{x_1}^{\al_1}\cdots \pa_{x_n}^{\al_n}$ for 
$\al=(\al_1,\ldots,\al_n)\in\bN_0^n$ and $\pa_{x_i}=\pa/\pa x_i$, 
$1\le i\le n$. We then associate to $T$ an $n$-parameter family of 
linear partial differential operators with {\em constant} coefficients 
$\{T_\by\}_{\by\in\bR^n}$ given by
$$
T_\by=\sum_{\al\in\bN_0^n}q_\al(\by)
\pa^\al
$$
for any fixed $\by=(y_1,\ldots,y_n)\in\bR^n$. Moreover, for all  
$\al=(\al_1,\ldots,\al_n)\in\bN_0^n$ we define an $n$-parameter family of 
polynomials 
$\{p_{\by,\al}\}_{\by\in\bR^n}$ by 
$$
p_{\by,\al}(\bx)=\sum_{\be\le \al}q_\be(\by)\bx^\be\in \bR[x_1,\ldots,x_n]   
$$
for each fixed $\by=(y_1,\ldots,y_n)\in\bR^n$, where 
$\bx^\be=x_1^{\be_1}\cdots x_n^{\be_n}$
for $\bx=(x_1,\ldots,x_n)$, $\be=(\be_1,\ldots,\be_n)\in\bN_0^n$, 
and we use the standard (product) 
partial order on $\bN_0^n$: $\be\le \al$ provided that $\be_i\le \al_i$, 
$1\le i\le n$. 
Let us also introduce an $n$-parameter family of real-valued 
maps $\{\cH_\by\}_{\by\in\bR^n}$ on the additive semigroup $\bN_0^n$ by 
setting 
$$
\cH_\by(\al)=\al!q_\al(\by),\quad \al=(\al_1,\ldots,\al_n)\in\bN_0^n,
$$
for any fixed $\by=(y_1,\ldots,y_n)\in\bR^n$, where as before 
$\al!=\al_1!\cdots\al_n!$. 

\begin{remark}\label{r-symb-m}
Clearly, $p_{\by,\al}\in\bR[x_1,\ldots,x_n,y_1,\ldots,y_n]$ for each 
$\al\in\bN_0^n$. By analogy with the terminology 
used in the univariate case (cf.~Remark \ref{r-term}) 
we call this $2n$ variable 
polynomial the $\al$-{\em truncated symbol} of the operator $T$. 
\end{remark}

Set 
$2\al=(2\al_1,\ldots,2\al_n)\in\bN_0^n$ whenever 
$\al=(\al_1,\ldots,\al_n)\in\bN_0^n$. The following theorem 
characterizes all linear operators in $\fL(n)$ that 
preserve the cone $\cN(n)$ of non-negative polynomials in $n$ variables.

\begin{theorem}\label{th-nn}
For $T\in\fL(n)$ as in \eqref{w-n} the following assertions are equivalent:
\begin{itemize}
\item[(1)] $T(\cN(n))\subseteq \cN(n)$; 
\item[(2)] $T_\by(\cN(n))\subseteq \cN(n)$ for all $\by\in\bR^n$;
\item[(3)] $\langle p_{\by,2\al},f\rangle \ge 0$ for all $\by\in\bR^n$, 
$\al\in\bN_0^n$, $f\in\cN(n)_{2\al}$, i.e., 
$p_{\by,2\al}\in\cN(n)_{2\al}^*$; 
\item[(4)] For any $\bs=(s_1,\ldots,s_n)\in\bR^n$ there exists 
$\mu_\bs\in\fM_n^+$ such that 
$\int\!\bt^\al d\mu_\bs(\bt)\in\bR[s_1,\ldots,s_n]$, $\al\in\bN_0^n$, and 
$\cH_\by(\al)=\int\!\bt^\al d\mu_\by(\bt)$, $\by\in\bR^n$, $\al\in\bN_0^n$, 
where $\bt^\al=t_1^{\al_1}\cdots t_n^{\al_n}$ for 
$\bt=(t_1,\ldots,t_n)\in\bR^n$. 
The family of operators $\{T_\by\}_{\by\in\bR^n}$ is then given by the 
convolutions
$$
T_\by(f)(\bx)=\int\! f(\bt+\bx)d\mu_\by(\bt),\quad f\in\bR[x_1,\ldots,x_n],\,
\by\in\bR^n.
$$
\end{itemize}
Furthermore, each of these conditions implies:
\begin{itemize}
\item[(5)] The map $\cH_\by:\bN_0^n\to \bR$ is a
positive semidefinite kernel on $\bN_0^n\times\bN_0^n$ 
for each $\by\in\bR^n$, i.e., the Hankel 
matrix
$
\big(\cH_\by(\al^i+\al^j)\big)_{i,j=1}^{m}
$
is positive semidefinite for any $m\in\bN$ and 
$\{\al^1,\ldots,\al^m\}\subset \bN_0^n$.
\end{itemize}
\end{theorem}

\begin{proof}
For each $\al=(\al_1,\ldots,\al_n)\in\bN_0^n$ the cone 
$\cN(n)_\al$ is invariant under shift operators $e^{\ba\cdot\pa}$, 
$\ba=(a_1,\ldots,a_n)\in\bR^n$, where 
$\ba\cdot\pa=a_1\pa/\pa x_1+\cdots+a_n\pa/\pa x_n$. Moreover, 
one has 
$$
T_\by(f)(\ba)=\sum_{\be\le \al}q_\be(\by)\pa^{\be}f(\ba)=
\langle p_{\by,\al}, e^{\ba\cdot\pa}f\rangle, \quad 
f\in\bR_\al[x_1,\ldots,x_n].
$$ 
Note also that if $f\in\cN(n)$ then 
$\deg_{x_i}(f)\equiv 0\bmod 2$, $1\le i\le n$. 
The equivalences (1) $\Leftrightarrow$ (2) $\Leftrightarrow$ (3) follow 
readily from 
these facts and straightforward extensions of the arguments in the proof of 
Theorem \ref{th-s} to the multivariate case. 

Next, since $T_\by(f)(\bo)=\langle p_{\by,\be},f\rangle$ if 
$f\in\bR_\be[x_1,\ldots,x_n]$, $\be\in\bN_0^n$, Taylor's formula yields 
(4) $\Rightarrow$ (3). To prove the converse, 
for $\by\in\bR^n$ we let $L_\by:\bR[x_1,\ldots,x_n]\to\bR$ 
be the linear functional determined by $L_\by(\bx^\be)=\cH_\by(\be)$, 
$\be\in\bN_0^n$. It follows that 
$$
L_\by(f)=\langle p_{\by,\be},f\rangle,\quad f\in\bR_\be[x_1,\ldots,x_n],\,
\be\in\bN_0^n.
$$
In particular, if (3) holds then $L_\by(f)\ge 0$ for all $f\in\cN(n)_{2\al}$, 
$\al\in\bN_0^n$. 
From Haviland's theorem (cf.~\S \ref{s-in}) and the fact 
that $\cH_\by(\be)$ is a polynomial in $\by$ for each $\be\in\bN_0^n$ we 
deduce that (3) $\Rightarrow$ (4).

Let now $m\in\bN$, $\{\al^1,\ldots,\al^m\}\subset \bN_0^n$, $c_i\in\bR$, 
$1\le i\le m$, and set 
$$
g(\bx)=\sum_{i=1}^{m}c_i\bx^{\al^i}\in\bR[x_1,\ldots,x_n].
$$ 
If $\al\in\bN_0^n$ is such that $\al^i+\al^j\le \al$, $1\le i,j\le m$, 
one has the identity
\begin{equation}\label{eq-extra}
\langle p_{\by,\al},g^2\rangle=\sum_{i,j=1}^m\cH_\by(\al^i+\al^j)c_ic_j.
\end{equation}
Since $g^2\in\cN(n)$ this proves that (3) $\Rightarrow$ (5).
\end{proof}

The next theorem gives a characterization of all linear operators in $\fL(n)$ 
that preserve the cone $\cP(n)$ of positive polynomials in $n$ variables. 

\begin{theorem}\label{th-pm}
For $T\in\fL(n)$ as in \eqref{w-n} the following assertions are equivalent:
\begin{itemize}
\item[(1)] $T(\cP(n))\subseteq \cP(n)$; 
\item[(2)] $T_\by(\cP(n))\subseteq \cP(n)$ for all $\by\in\bR^n$;
\item[(3)] $\langle p_{\by,2\al},f\rangle > 0$ for all $\by\in\bR^n$, 
$\al\in\bN_0^n$, $f\in\cP(n)_{2\al}$;
\item[(4)] For any $\bs=(s_1,\ldots,s_n)\in\bR^n$ there exists 
$\mu_\bs\in\fM_n^+$ whose support is not contained in a proper real algebraic 
variety\footnote{In other words, $\supp(\mu_\bs)$ is not a subset of 
$f^{-1}(0)$ for some $f\in\bR[x_1,\ldots,x_n]\setminus \{0\}$.} such that 
for all $\al\in\bN_0^n$ one has 
$\int\!\bt^\al d\mu_\bs(\bt)\in\bR[s_1,\ldots,s_n]$ and 
$\cH_\by(\al)=\int\!\bt^\al d\mu_\by(\bt)$, $\by\in\bR^n$. 
The family of operators $\{T_\by\}_{\by\in\bR^n}$ is then given by the 
convolutions
$$
T_\by(f)(\bx)=\int\! f(\bt+\bx)d\mu_\by(\bt),\quad f\in\bR[x_1,\ldots,x_n],\,
\by\in\bR^n.
$$
\end{itemize}
Furthermore, each of these conditions implies:
\begin{itemize}
\item[(5)] The map $\cH_\by:\bN_0^n\to \bR$ is a
positive definite kernel on $\bN_0^n\times\bN_0^n$ 
for each $\by\in\bR^n$, i.e., the Hankel 
matrix
$
\big(\cH_\by(\al^i+\al^j)\big)_{i,j=1}^{m}
$
is positive definite for any $m\in\bN$ and 
$\{\al^1,\ldots,\al^m\}\subset \bN_0^n$.
\end{itemize}
\end{theorem}

\begin{proof}
Since $\cP(n)_{2\al}$ is invariant under shift operators $e^{\ba\cdot\pa}$, 
$\ba\in\bR^n$, the arguments in the proof of Theorem \ref{th-nn} 
carry over {\em mutatis mutandis} 
to the present situation so we will not repeat them. 
The only novelty concerns the extra condition in (4) above 
on the support of the
 measures $\mu_\bs$ (this condition does not appear in 
Theorem \ref{th-nn}~(4)) for which we refer to e.g.~Exercise 3.14 in 
\cite[Chap.~6]{BCR}.
\end{proof}

\begin{remark}
Elementary homotopy arguments show that $\cE(n)=\cP(n)\cup(-\cP(n))$ and 
if $T\in\fL(n)$ then 
$T(\cE(n))\subseteq \cE(n)$ if and only if either 
$T(\cP(n))\subseteq \cP(n)$ or $T(\cP(n))\subseteq -\cP(n)$. Therefore,  
Theorem \ref{th-pm} also yields 
a description of all linear operators in $\fL(n)$ preserving 
the set $\cE(n)$ of elliptic polynomials in $n$ variables.
\end{remark}

The above theorems solve the multivariate analogs of the ``soft'' 
(transcendental/unbounded 
degree) classification questions stated in Problem \ref{prob2}. It is clear 
that one gets ``hard'' (algebraic/bounded 
degree) theorems in similar fashion. More precisely,
given $\al\in\bN_0^n$ and a linear operator $T:\bR_{2\al}[x_1,\ldots,x_n]\to
\bR[x_1,\ldots,x_n]$ it follows from the proofs of 
Theorems \ref{th-nn}--\ref{th-pm} that
\begin{equation*}
\begin{split}
T(\cN(n)_{2\al})\subseteq \cN(n) &\Leftrightarrow 
T_\by(\cN(n)_{2\al})\subseteq \cN(n) \,\text{ for all }\,\by\in\bR^n \\
&\Leftrightarrow 
\langle p_{\by,2\al},f\rangle \ge 0 \,\text{ for all }\,\by\in\bR^n, 
f\in\cN(n)_{2\al} 
\end{split}
\end{equation*}
(i.e., $p_{\by,2\al}\in\cN(n)_{2\al}^*$ for all $\by\in\bR^n$) and
\begin{equation*}
\begin{split}
T(\cP(n)_{2\al})\subseteq \cP(n) &\Leftrightarrow 
T_\by(\cP(n)_{2\al})\subseteq \cP(n) \,\text{ for all }\,\by\in\bR^n \\
&\Leftrightarrow 
\langle p_{\by,2\al},f\rangle > 0 \,\text{ for all }\,\by\in\bR^n, 
f\in\cP(n)_{2\al}.
\end{split}
\end{equation*}
These equivalences answer the multivariate analogs of the classification 
questions raised in Problem \ref{prob1}.

\section{Related Results and Further Directions}\label{s-rem}

\noindent
{\bf 1.}~It is interesting to compare 
the classification theorems established here with those obtained 
in \cite{BB-LY,BB-A,BB-P} for linear operators preserving real-rooted 
polynomials or their multivariate analogs. A common feature is the use of 
appropriately defined operator symbols, see Remarks \ref{r-term}, 
\ref{r-algs}, and \ref{r-symb-m}. However, the truncated symbols 
introduced here are employed in a different way than the symbols used in 
\cite{BB-LY,BB-A}.
Note for instance that if $T\in\fL(\cP)$ is as in 
\eqref{w-1} then for each $m\in\bN_0$ its $2m$-truncated symbol 
$p_{y,2m}(x)$ is a positive polynomial in variables $x,y$ (since all
even degree truncations of the exponential series are positive) but the 
converse statement is not true, as one can see from 
e.g.~Corollary \ref{c-fo} (b) below. By contrast, the main results 
in \cite{BB-A} essentially assert that a linear operator 
preserves the set of univariate real-rooted polynomials if and only if its 
symbol is a ``real-rooted'' polynomial in two variables (in some appropriate 
sense). 

We should also note that there are plenty of linear differential 
operators of positive order in the Weyl algebra $\cA_1(\bR)$ which 
preserve real-rootedness 
(these are characterized in \cite{BB-A,BB-P}) whereas 
no such operators preserve positive, non-negative or elliptic polynomials 
(cf.~Corollary \ref{c-fo}).

The equivalent conditions in our main results successively describe linear 
operators preserving non-negativity by means of linear differential operators 
with constant coefficients, Fischer-Fock dual cones, 
positive semidefinite Hankel forms, and convolution operators induced 
by non-negative measures (Theorems \ref{th-s} and \ref{th-nn}). 
The Fischer-Fock product was also used in Theorem 1.11 of \cite{BB-P} 
which asserts that 
an operator $T\in \cA_1(\bR)$ preserves real-rootedness if and only if its 
Fischer-Fock adjoint $T^*$ has the same property. Note that this duality 
fails for linear 
operators preserving positive, non-negative or elliptic polynomials since 
multiplication by a positive polynomial clearly 
preserves the sets $\cS$, $\cP$, $\cE$ but differential operators of positive 
order in $\cA_1(\bR)$ do not. Nevertheless, Theorem \ref{th-s} (3) and 
Theorem \ref{th-nn} (3) show that appropriate Fischer-Fock dualities 
(in one or 
several variables) do hold in the present context as well.

\bigskip

\noindent
{\bf 2.}~The classification theorems in \S \ref{s-trans}--\ref{s-alg} 
yield in particular 
the following improvements 
of the main results in \cite{BGS} (Theorem A, Corollary, and 
Theorem B in {\em op.~cit.}).

\begin{corollary}\label{c-fo}
In the above notations one has:
\begin{itemize}
\item[(a)] If $T\in\fL$ is an element of the Weyl algebra $\cA_1(\bR)$ of
order 
$d\ge 1$ then for any even integer $\ell>d$ one has 
$T(V_\ell)\setminus V\neq\emptyset$ 
whenever $V\in \{\cP,\cS,\cE\}$;
\item[(b)] There are no linear ordinary differential operators with polynomial 
coefficients of finite 
positive order in $\fL(\cP)\cup \fL(\cS)\cup \fL(\cE)$; 
\item[(c)] If $T\in\fL\setminus\{0\}$ is a linear 
ordinary differential operator with constant coefficients then 
$T\in\fL(\cP)\Leftrightarrow T\in\fL(\cS)\Leftrightarrow 
T\in\fL(\cE)\Leftrightarrow T$ is a convolution operator of the form
$$
T(f)(x)=\int\! f(t+x)d\mu(t),\quad f\in \bR[x],
$$
for some $\mu\in\fM_1^+$.
\end{itemize}
\end{corollary}

\begin{proof}
To prove (a) it is enough to show that $T\notin \fL_\ell(V)$ 
for the smallest even integer $\ell>d$ and 
$V=\cS$ (the arguments for $V=\cP$ and $V=\cE$ are similar). Let 
$$
T=\sum_{i=0}^{d}q_i(x)D^i,\quad q_i\in\bR[x],\,i\in\bN_0,\,q_d\not\equiv 0,
$$
and $y\in\bR$ be such that $q_d(y)\neq 0$. If $d$ is odd set $q_{d+1}(y)=0$. 
Then the 
determinant of the lower right $2\times 2$ principal submatrix of $\cH_{y,d+1}$
is $-q_d(y)^2$ so $\cH_{y,d+1}$ is not positive 
semidefinite and by Theorem \ref{th-sd} 
$T\notin \fL_{d+1}(\cS)$. If $d$ is even set 
$q_{d+1}(y)=q_{d+2}(y)=0$ and suppose that $T\in\fL_d(\cS)$ 
(if $T\notin\fL_d(\cS)$ there is nothing to prove since 
$\cS_d\subset \cS_{d+2}$). By Theorem \ref{th-sd} 
$\cH_{y,d}$ is positive semidefinite hence $q_d(y)>0$. 
The 
determinant of the lower right $3\times 3$ principal submatrix of $\cH_{y,d+2}$
equals $-q_d(y)^3$ and thus $\cH_{y,d+2}$ cannot be positive semidefinite. 
By Theorem \ref{th-sd} again we conclude that $T\notin\fL_{d+2}(\cS)$.

Clearly, (a) $\Rightarrow$ (b)  
while (c) follows from Theorem \ref{th-s} and the fact that 
if $T\in\fL(\cS)$ and $T(1)\equiv 0$ then $T\equiv 0$ (cf.~\S \ref{s-prel}). 
\end{proof}

\begin{remark}
Corollary \ref{c-fo} (a) asserts that a linear 
ordinary differential operator of
order $d\ge 1$ with polynomial coefficients cannot preserve positivity, 
non-negativity, or ellipticity 
when acting on $\bR_\ell[x]$, where $\ell$ is any even 
integer greater than $d$. In particular, this implies part (b) stating 
that there are no 
linear ordinary differential operators of finite positive order (acting on 
$\bR[x]$) that preserve 
positivity, non-negativity, or ellipticity. 
Finally, part (c) asserts that a non-trivial 
linear ordinary differential operator with constant coefficients (acting on 
$\bR[x]$) preserves each of the sets $\cP$, $\cS$, $\cE$ provided that it 
preserves at least one of them and any such operator is a convolution with 
a non-negative measure with all finite moments.
\end{remark}

\medskip

\noindent
{\bf 3.~}From Theorems \ref{th-nn}--\ref{th-pm} we deduce that any linear 
partial differential operator with constant coefficients that preserves  
non-negativity or positivity is actually a convolution operator induced by a 
non-negative measure with all finite moments. More precisely, we have the 
following multivariate analog of Corollary \ref{c-fo} (c).

\begin{corollary}\label{c-2}
Let $T\in\fL(n)$ be a linear partial differential operator with constant 
coefficients.
\begin{itemize}
\item[(i)] The following assertions are equivalent:
\begin{itemize}
\item[(a)] $T(\cN(n))\subseteq \cN(n)$; 
\item[(b)] There exists $\mu\in\fM_n^+$ such that $T$ may be represented as 
the convolution
$$
T(f)(\bx)=\int\! f(\bt+\bx)d\mu(\bt),\quad f\in\bR[x_1,\ldots,x_n].
$$
\end{itemize}
\item[(ii)] The following assertions are also equivalent:
\begin{itemize}
\item[(c)] $T(\cP(n))\subseteq \cP(n)$; 
\item[(d)] There exists $\mu\in\fM_n^+$ whose support is not contained in a 
proper real algebraic 
variety such that $T$ may be represented as the convolution
$$
T(f)(\bx)=\int\! f(\bt+\bx)d\mu(\bt),\quad f\in\bR[x_1,\ldots,x_n].
$$
\end{itemize}
\end{itemize}
\end{corollary}

A similar corollary follows for diagonal operators in the standard monomial 
basis of $\bR[x_1,\ldots,x_n]$, i.e., operators $T\in\fL(n)$ given by 
$T(\bx^\be)=\la_\be \bx^\be$, where $\la_\be\in\bR$, $\be \in\bN_0^n$. It is 
easy to see that such an operator may be uniquely represented as
$$
T=\sum_{\al\in\bN_0^n}a_\al\bx^\al \pa^\al,\quad a_\al\in\bR,\,\al\in\bN_0^n,
$$
and that $\la_\be=\sum_{\al\le \be}(\be)_\al a_\al$ for $\be \in\bN_0^n$, 
where as usual $(\be)_\al=(\be_1)_{\al_1}\cdots (\be_n)_{\al_n}$ and 
$(m)_k=m(m-1)\cdots (m-k+1)$ whenever $k,m\in\bN_0$ with $k\le m$. Elementary 
computations then yield the following implication:
$$
\al!a_\al=\int \bt^\al d\mu(\bt),\,\al\in\bN_0^n\,\,\,\Longrightarrow \,\,\,
\la_\be=\int (\bt+\bone)^\be d\mu(\bt)=\int \bs^\be d\nu(\bs),\, 
\be\in \bN_0^n,
$$
where $\bone$ denotes the all ones vector and $\bs=\bt+\bone$ 
(note in particular that for $\mu,\nu$ as above one has 
$\mu\in\fM_n^+\Leftrightarrow \nu\in\fM_n^+$). Hence, if the condition in the 
left-hand side of the above implication holds we deduce that the diagonal 
operator $T$ is given by
$$
T(f)(\bx)=\int f(\bs\bx)d\nu(\bs),\quad f\in\bR[x_1,\ldots,x_n],
$$
where $\bs\bx=(s_1x_1,\ldots,s_nx_n)$. The next statement is therefore an 
immediate consequence of these arguments and 
Theorems \ref{th-nn}--\ref{th-pm}.

\begin{corollary}\label{c-3}
Let $T\in\fL(n)$ be a diagonal operator given by 
$T(\bx^\be)=\la_\be \bx^\be$, $\la_\be\in\bR$, $\be \in\bN_0^n$.
\begin{itemize}
\item[(i)] The following assertions are equivalent:
\begin{itemize}
\item[(a)] $T(\cN(n))\subseteq \cN(n)$; 
\item[(b)] There exists $\nu\in\fM_n^+$ such that $T$ may be represented as 
$$
T(f)(\bx)=\int f(\bs\bx)d\nu(\bs),\quad f\in\bR[x_1,\ldots,x_n].
$$
\end{itemize}
\item[(ii)] The following assertions are also equivalent:
\begin{itemize}
\item[(c)] $T(\cP(n))\subseteq \cP(n)$; 
\item[(d)] There exists $\nu\in\fM_n^+$ whose support is not contained in a 
proper real algebraic 
variety such that $T$ may be represented as 
$$
T(f)(\bx)=\int f(\bs\bx)d\nu(\bs),\quad f\in\bR[x_1,\ldots,x_n].
$$
\end{itemize}
\end{itemize}
\end{corollary}

\begin{remark}\label{r-hom}
Let $T\in\fL(n)$ be either a partial differential operator with 
constant coefficients or a diagonal operator. Note that $T(1)$ is then 
a constant polynomial and suppose that $T(\cN(n))\subseteq \cN(n)$. 
If $T(1)=0$ it 
follows from Corollary \ref{c-2} (i) (respectively, Corollary \ref{c-3} (i)) 
that 
the the non-negative measure $\mu$ (respectively, $\nu$) has total mass zero
and then $T\equiv 0$. Therefore, if $T\not\equiv 0$ then $T(1)>0$ and so if 
$f\in\cP(n)$ one gets $T(f)=T(f-a)+aT(1)\in\cP(n)$, where 
$a>0$ is such that $f-a\in\cP(n)$ (clearly, such $a$ exists for any 
$f\in\cP(n)$). We conclude that if $T\in\fL(n)\setminus \{0\}$ is as in 
Corollary \ref{c-2} or Corollary \ref{c-3} 
then $T(\cN(n))\subseteq \cN(n)\Leftrightarrow 
T(\cP(n))\subseteq \cP(n)$.   
\end{remark}

\medskip

\noindent
{\bf 4.~}If $n\ge 2$ one has 
in general $\cS(n)\subsetneq \cN(n)$ and it is 
natural to ask 
for a characterization of the set $\{T\in\fL(n): T(\cS(n))\subseteq \cN(n)\}$ 
(obviously, this set contains the semigroup 
of all $T\in\fL(n)$ such that $T(\cN(n))\subseteq \cN(n)$ that we already 
described in Theorem \ref{th-nn}). 
The following theorem provides an answer to this question. 

\begin{theorem}\label{th-sn}
For $n\ge 2$ and $T\in\fL(n)$ as in \eqref{w-n} the following are equivalent:
\begin{itemize}
\item[(1)] $T(\cS(n))\subseteq \cN(n)$; 
\item[(2)] $T_\by(\cS(n))\subseteq \cN(n)$ for all $\by\in\bR^n$;
\item[(3)] $\langle p_{\by,2\al},f\rangle \ge 0$ for all $\by\in\bR^n$, 
$\al\in\bN_0^n$, $f\in\cS(n)_{2\al}$, i.e., $p_{\by,2\al}\in\cS(n)_{2\al}^*$; 
\item[(4)] The map $\cH_\by:\bN_0^n\to \bR$ is a
positive semidefinite kernel on $\bN_0^n\times\bN_0^n$ 
for each $\by\in\bR^n$, i.e., the Hankel 
matrix
$
\big(\cH_\by(\al^i+\al^j)\big)_{i,j=1}^{m}
$
is positive semidefinite for any $m\in\bN$ and 
$\{\al^1,\ldots,\al^m\}\subset \bN_0^n$.
\end{itemize}
\end{theorem}

\begin{proof}
Since $\cS(n)=\cup_{\al\in\bN_0^n}\cS(n)_{2\al}$ and each cone 
$\cS(n)_{2\al}$ is invariant under shift operators   
$e^{\ba\cdot\pa}$, $\ba\in\bR^n$, the implications (1) $\Leftrightarrow$ 
(2) $\Leftrightarrow$ (3) $\Rightarrow$ (4) follow 
as in the proof of Theorem~\ref{th-nn}. To show that (4) $\Rightarrow$ (3) 
note that $f\in \cS(n)_{2\al}$ if and only if $f=\sum_{j=1}^{r}g_j^2$ for some 
$r\in\bN$ and $g_j\in\bR_{\al}[x_1,\ldots,x_n]$, $1\le j\le r$. Given 
$\al\in\bN_0^n$ let $m=m(\al)$ be the cardinality of 
$\{\be\in\bN_0^n:\be\le \al\}$ (which is clearly finite) and order this set 
lexicographically as $\{\al^1,\ldots,\al^m\}$. Any 
$g\in\bR_{\al}[x_1,\ldots,x_n]$ may then be written as
$$
g(\bx)=\sum_{i=1}^{m}c_i\bx^{\al^i},\quad c_i\in\bR,\,1\le i\le m,
$$
so the desired conclusion follows from \eqref{eq-extra}. 
\end{proof}

In view of Theorem \ref{th-sn} and Corollary \ref{c-2} one may ask whether the 
following analog of Hilbert's 17th problem holds, cf.~\cite[Problem 17]{BBCV}. 

\begin{question}
Is it true that for any $p\in\cN(n)$ there exist $f\in\cS(n)$ and a linear 
partial
differential operator with constant coefficients $T\in\fL(n)$ such that 
$T(\cS(n))\subseteq \cN(n)$ and $T(f)=p$?
\end{question}


\begin{thebibliography}{99}

\bibitem{Ak}
N.~I.~Akhiezer, The Classical Moment Problem. 
Oliver \& Boyd, Edinburgh \& London, 1965. 

\bibitem{BCR}
C.~Berg, J.~P.~R.~Christensen, P.~Ressel, Harmonic analysis on semigroups. 
Graduate Texts in 
Mathematics, Vol.~100. Springer-Verlag, New York, 1984.

\bibitem{BB-LY}
J.~Borcea, P.~Br\"and\'en, {\em The Lee-Yang 
and P\'olya-Schur programs}, Parts I, II,  
preprints arXiv:0809.0401, arXiv:0809.3087.

\bibitem{BB-A}
J.~Borcea, P.~Br\"and\'en, {\em P\'olya-Schur master theorems for 
circular domains and their boundaries}, Ann. of Math., to appear; 
http://annals.math.princeton.edu/.

\bibitem{BB-D}
J.~Borcea, P.~Br\"and\'en, {\em Applications of stable 
polynomials to mixed determinants: Johnson's conjectures, unimodality, and 
symmetrized Fischer products}, Duke Math. J. {\bf 143} (2008), 205--223.

\bibitem{BB-P}
J.~Borcea, P.~Br\"and\'en, {\em Multivariate P\'olya-Schur classification 
problems in the Weyl algebra}, preprint arXiv:math/0606360.

\bibitem{BBCV}
J.~Borcea, P.~Br\"and\'en, G.~Csordas, V.~Vinnikov, {\em P\'olya-Schur-Lax 
problems: hyperbolicity and stability preservers}, 
http://www.aimath.org/pastworkshops/polyaschurlax.html.

\bibitem{BBL}
J.~Borcea, P.~Br\"and\'en, T.~M.~Liggett, {\em Negative dependence and the 
geometry of polynomials}, J. Amer. Math. Soc., to appear; 
http://www.ams.org/journals/jams/.

\bibitem{BGS}
J.~Borcea, A.~Guterman, B.~Shapiro, {\em Preserving positive polynomials and 
beyond}, preprint arXiv:0801.1749.

\bibitem{C}
T.~Carleman, {\em Sur le probl\`eme des moments}, 
C. R. Acad. Sci. Paris {\bf 174} (1922), 1680--1682.

\bibitem{CC} 
T.~Craven, G.~Csordas,
{\em Composition theorems, multiplier sequences and complex zero
decreasing sequences},  in  ``Value Distribution Theory
and Its Related Topics'', G. Barsegian, I. Laine, 
C. C. Yang (Eds.), pp.~131--166, Kluwer Press, 2004.

\bibitem{CC1}
T.~Craven, G.~Csordas, {\em Problems and theorems in the theory of multiplier 
sequences}, Serdica Math. J. {\bf 22} (1996), 515--524. 

\bibitem{CC2}
T.~Craven, G.~Csordas, {\em Complex zero decreasing sequences}, Methods 
Appl. Anal. {\bf 2} (1995), 420--441. 

\bibitem{Ham}
H.~Hamburger, {\"Uber eine Erweiterung des Stieltjesschen Momentenproblems}, 
Parts I, II, III, Math. Ann. {\bf 81} (1920), 235--319; ibid. {\bf 82} (1921), 
20--164, 168--187.

\bibitem{Hav}
E.~K.~Haviland, {\em On the momentum problem for distribution functions in 
more than one dimension}, Parts I, II, Amer. Math. J {\bf 57} (1935), 
562--568; ibid. {\bf 58} (1936), 164--168.

\bibitem{H-P}
J.~W.~Helton, M.~Putinar, {\em Positive polynomials in scalar and matrix 
variables, the spectral theorem, and optimization}, in ``Operator theory, 
structured matrices, and dilations'', pp.~229--306, Theta Ser. Adv. Math., 7, 
Theta, Bucharest, 2007.

\bibitem{H}
A.~Hurwitz, {\em \"Uber definite Polynome}, Math. Ann. {\bf 73} (1913), 
173--176.

\bibitem{Ka}
S.~Karlin, Total Positivity. Vol.~I, Stanford Univ. Press, Stanford, CA, 1968.

\bibitem{PS}
G.~P\'olya, I.~Schur,  {\em \"Uber zwei Arten von
Faktorenfolgen in der Theorie der algebraischen Gleichungen}, J. Reine
Angew. Math. {\bf 144} (1914), 89--113.

\bibitem{PSz}
G.~P\'olya, G.~Szeg\H{o}, Problems and Theorems in Analysis. Vol.~II. 
Reprint of the 1976 English translation. Classics in Mathematics. 
Springer-Verlag, Berlin, 1998. 

\bibitem{PD} 
A.~Prestel, C.~N.~Delzell, Positive Polynomials. From Hilbert's 17th Problem 
to Real Algebra. Springer Monographs in Mathematics, Springer-Verlag, 
Berlin, 2001.

\bibitem{Pu-S}
M.~Putinar, C.~Scheiderer, {\em Multivariate moment problems: geometry and 
indeterminateness},  Ann. Sc. Norm. Super. Pisa Cl. Sci. (5) {\bf 5}  
(2006), 137--157. 

\bibitem{Pu-Sm}
M.~Putinar, K.~Schm\"udgen, {\em Multivariate determinateness}, preprint 
arXiv:0810.0840.

\bibitem{Pu-V}
M.~Putinar, F.-H.~Vasilescu, 
{\em A uniqueness criterion in the multivariate moment problem}, 
Math. Scand. {\bf 92} (2003), 295--300. 

\bibitem{Re} 
R.~Remak, {\em Bemerkung zu Herrn Stridsbergs Beweis des Waringschen 
Theorems}, Math. Ann. {\bf 72} (1912), 153--156. 

\bibitem{R}
B.~Reznick, {\em Sums of even powers of real linear forms}.  
Mem. Amer. Math. Soc. {\bf 96} (1992), no. 463.

\bibitem{Ri}
M.~Riesz, {\em Sur le probl\`eme des moments. Troisi\`eme Note}, Arkiv 
f\"or matematik, astronomi och fysik {\bf 17} (1923), no. 16, 1--52.

\bibitem{ST}
J.~A.~Shohat, J.~D.~Tamarkin, The Problem of Moments. Amer. Math. Soc., 
Providence, R.I., 1943.

\bibitem{W} 
D.~V.~Widder, The Laplace Transform. Princeton Math. Series Vol.~6, 
Princeton Univ. Press, Princeton, NJ, 1941.

\end{thebibliography}
\end{document}